\documentclass[12pt]{amsart}
\usepackage{amssymb}
\usepackage{hyperref}
\theoremstyle{plain}
\newtheorem{theorem}{Theorem}[section]
\newtheorem{proposition}[theorem]{Proposition}
\newtheorem{corollary}[theorem]{Corollary}
\newtheorem{lemma}[theorem]{Lemma}

\theoremstyle{definition}
\newtheorem{definition}[theorem]{Definition}
\newtheorem{example}[theorem]{Example}
\newtheorem{remark}[theorem]{Remark}

\theoremstyle{remark}

\numberwithin{equation}{section}

\newcommand{\N}{\mathbb N}
\newcommand{\Z}{\mathbb Z}

\newcommand{\C}{\mathbb C}

\newcommand{\SO}{\operatorname{SO}}
\newcommand{\SU}{\operatorname{SU}}

\newcommand{\Spin}{\operatorname{Spin}}

\newcommand{\Hom}{\operatorname{Hom}}

\newcommand{\tr}{\operatorname{tr}}

\newcommand{\ba}{\backslash}

\newcommand{\String}[1]{\mathcal{S}(#1)} %{\mathcal{S}_{{#1}}} %versión anterior
\newcommand{\NN}{{|\Phi^+|}}
\newcommand{\NNg}{{|\Phi_\gamma^+|}}
\newcommand{\PP}{\mathcal{P}}

\newcommand{\dir}{\omega}

\title[Strong multiplicity one theorems]
{Strong multiplicity one theorems for locally homogeneous spaces of compact type}

\author{Emilio A. Lauret \and Roberto J. Miatello}

\address{Instituto de Matem\'atica (INMABB), Departamento de Matem\'atica, Universidad Nacional del Sur (UNS)-CONICET, Bah\'ia Blanca B8000CPB, Argentina.}
\email{emilio.lauret@uns.edu.ar}

\address{CIEM--FaMAF (CONICET), Universidad Nacional de C\'ordoba, Medina Allende s/n, Ciudad Universitaria, 5000 C\'ordoba, Argentina.}
\email{miatello@famaf.unc.edu.ar}

\subjclass[2010]{22C05, 22E46.}
\keywords{Strong multiplicity one theorem, right regular representation, representation equivalent}

\thanks{This research was supported by grants from CONICET, FONCyT and SeCyT. The first named author was supported by the Alexander von Humboldt Foundation}
\date{November 2019}

\begin{document}

\begin{abstract}
Let $G$ be a compact connected semisimple Lie group, let $K$ be a closed subgroup of $G$, let $\Gamma$ be a finite subgroup of $G$, and let $\tau$ be a finite dimensional  representation of $K$. 
For $\pi$ in the unitary dual $\widehat G$ of $G$, denote by $n_\Gamma(\pi)$ its multiplicity in $L^2(\Gamma\backslash G)$.
				
We prove a strong multiplicity one theorem in the spirit of Bhagwat and Rajan, for the $n_\Gamma(\pi)$ for $\pi$ in the set $\widehat G_\tau$ of irreducible $\tau$-spherical representations of $G$. 
More precisely, for $\Gamma$ and $\Gamma'$ finite subgroups of $G$, we prove that if $n_{\Gamma}(\pi)= n_{\Gamma'}(\pi)$ for all but finitely many $\pi\in \widehat G_\tau$, then $\Gamma$ and $\Gamma'$ are $\tau$-representation equivalent, that is, $n_{\Gamma}(\pi)=n_{\Gamma'}(\pi)$ for all $\pi\in \widehat G_\tau$. 

Moreover, when $\widehat G_\tau$ can be written as a finite union of strings of representations, we prove a finite version of the above result. 
For any finite subset $\widehat {F}_{\tau}$ of $\widehat G_{\tau}$ verifying some mild conditions, the values of the $n_\Gamma(\pi)$ for $\pi\in\widehat  F_{\tau}$ determine the $n_\Gamma(\pi)$'s for all $\pi \in \widehat G_\tau$.
In particular, for two finite subgroups $\Gamma$ and $\Gamma'$ of $G$, if $n_\Gamma(\pi) = n_{\Gamma'}(\pi)$ for all $\pi\in \widehat F_{\tau}$ then the equality holds for every $\pi \in \widehat G_\tau$. 
We use algebraic methods involving generating functions and some facts from the representation theory of $G$. 
\end{abstract} 
	
\maketitle

\section{Introduction}
%\footnote{recordar fijar la fecha antes de subirlo a arxiv}
Let $G$ be a Lie group and $\Gamma$ a discrete cocompact subgroup of $G$. 
The right regular representation $R_\Gamma$ of $G$ on $L^2(\Gamma\ba G)$  decomposes as a discrete direct sum of unitary irreducible representations $(\pi, V_\pi)$ of $G$ occurring with finite multiplicity. 
That is,
\begin{equation}\label{eq1:decomprightregrep}
L^2(\Gamma\ba G) \simeq \bigoplus_{\pi\in\widehat{G}} n_\Gamma(\pi)\,V_\pi,
\end{equation}
with $n_\Gamma(\pi)\in\N_0:=\N\cup\{0\}$ for any $\pi$ in the unitary dual  $\widehat G$ of $G$.

Now, let $K$ be a compact subgroup of $G$ and let $(\tau, V_\tau)$ be a finite dimensional representation of $K$.
A unitary representation $(\pi, V_\pi)$ of $G$ will be called \emph{$\tau$-spherical} if $\Hom_K(V_\tau,V_\pi)\neq 0$.
Let $\widehat G_{\tau}$ denote the set of $\tau$-spherical irreducible representations of $G$.
The Hilbert space 
\begin{equation}\label{eq1:decomprightregrep_tau}
L^2(\Gamma\ba G)_\tau :=  \bigoplus_{\pi\in\widehat{G}_\tau} n_\Gamma(\pi)\,V_\pi
\end{equation}
defines a unitary representation of $G$. 
	
Two discrete cocompact subgroups $\Gamma$ and $\Gamma'$ of $G$ are said to be \emph{representation equivalent} (resp.\ \emph{$\tau$-representation equivalent}) in $G$ if the representations $L^2(\Gamma\ba G)$ and $L^2(\Gamma'\ba G)$ (resp.\ $L^2(\Gamma\ba G)_\tau$ and $L^2(\Gamma'\ba G)_\tau$) are unitarily equivalent, that is, if $n_{\Gamma}(\pi)=n_{\Gamma'}(\pi)$ for every $\pi\in\widehat{G}$ (resp.\ $\pi\in\widehat{G}_\tau$).
	
C.~Bhagwat and C.S.~Rajan~\cite{BhagwatRajan11} studied spectral analogues of the so called \emph{strong multiplicity one theorems}.
They showed that if $\Gamma$ and $\Gamma'$ are discrete cocompact subgroups in a semisimple Lie group $G$ 
such that $n_{\Gamma}(\pi)=n_{\Gamma'}(\pi)$ for all but finitely many $\pi\in\widehat G$, then $\Gamma$ and $\Gamma'$ are representation equivalent in $G$. 
Furthermore, they proved a similar result for  $1_K$-spherical representations of $G$, when $G/K$ is a non-compact symmetric space. 
D.~Kelmer~\cite{Kelmer14} obtained refinements of the last result when $G/K$ has real rank one, for any finite-dimensional representation $\tau$ of $K$.
He replaced the finite set of exceptions by a possibly infinite set of sufficiently small density.
Both  Bhagwat-Rajan and Kelmer used analytic methods, the Selberg trace formula as a main tool.

The main goal of this article is to obtain strong multiplicity one type theorems in the case when $G$ is a compact semisimple Lie group.
We will use only algebraic methods and facts from the representation theory of compact Lie groups. 
Furthermore, we will not assume any restriction on $K$ nor that the discrete subgroups $\Gamma$ and $\Gamma'$ act freely on $G/K$, so the quotient spaces $\Gamma\ba G/K$ and $\Gamma'\ba G/K$ are compact good orbifolds.
In this context, $\Gamma$ is necessarily finite and, by Frobenius reciprocity, $n_\Gamma(\pi)$ coincides with the dimension of the subspace of $\Gamma$-invariant elements of $V_\pi$.
We will show that under certain conditions on $G$, $K$, $\tau$, a finite subset of the multiplicities $n_\Gamma(\pi)$ determines all multiplicities for $\pi \in \widehat G_\tau$.

A main tool for us will be the notion of {what we shall call a } \emph{string} of irreducible representations of $G$.
Fix a maximal torus of $G$ and a positive system in the associated root system. 
By the highest weight theorem, the irreducible representations of $G$ are in correspondence with the elements in the set of $G$-integral dominant weights $\PP^{+}(G)$.
For $\Lambda\in \PP^{+}(G)$, we denote by $\pi_{\Lambda}$ the  irreducible representation of $G$ with highest weight $\Lambda$.
A string is a sequence of representations of the form $\{\pi_{\Lambda_0+k\dir}:k\in \N_0\}$ with $\dir,\Lambda_0\in\PP^{+}(G)$.
We call $\dir$ the direction and $\Lambda_0$ the base of the string.

The usefulness of the notion of string of representations is that one can study the multiplicities $n_\Gamma(\pi_{\Lambda_0+k\dir})$ by means of the generating function 
\begin{equation}
F_{\dir,\Lambda_0,\Gamma}(z):=\sum_{k\geq0}  n_\Gamma(\pi_{\Lambda_0+k\dir})\, z^k.
\end{equation}
By using a version of the Weyl character formula, in Proposition~\ref{prop2:F_Gamma-rational} we prove that $F_{\dir, \Lambda_0,\Gamma}(z)$ is a rational function with denominator $(1- z^{|\Gamma|})^{\NN+1}$, where $\NN$ stands for the number of positive roots.

Let $q$ be any positive integer divisible by $|\Gamma|$.
As a consequence of the rational form of $F_{\dir,\Lambda_0,\Gamma}(z)$, Proposition~\ref{lem3:SMOTstrings} shows that if a finite subset $\mathcal A\subset \N_0$ satisfies that 
\begin{equation}\label{eq1:finitecondition}
|\mathcal A\cap (j+q\Z)| \geq \NN+1  \quad \text{for all } 0\leq j\leq q-1, 
\end{equation}
then the coefficients $n_\Gamma(\pi_{\Lambda_0+k\dir})$ for $k\in\mathcal A$ determine $n_\Gamma( \pi_{\Lambda_0+k\dir})$ for all $k\geq0$. 
For instance, \eqref{eq1:finitecondition} holds when $\mathcal A$ contains any interval of length $q(\NN+1)$.
A bonus in the present context is that the rationality of the generating function allows  to obtain an expression for any $n_\Gamma( \pi_{\Lambda_0+k\dir})$ as a linear combination of the $\{n_\Gamma(\pi_{\Lambda_0+k'\dir}): k'\in\mathcal A\}$.

As a corollary of Proposition~\ref{lem3:SMOTstrings} we obtain a strong multiplicity one theorem for strings of representations, valid with a possible infinite set of exceptions of sufficiently small density (see Corollary~\ref{cor3:density2Gammas}). 
We also prove strong multiplicity one theorems for $\tau$-spherical representations.

\begin{theorem}\label{thm1:SMOTtau}
Let $G$ be a compact connected semisimple Lie group, let $K$ be a closed subgroup of $G$ and let $\tau$ be a finite dimensional representation of $K$.
If $\Gamma$ and $\Gamma'$ are finite subgroups of $G$ such that $n_{\Gamma}(\pi)=n_{\Gamma'}(\pi)$ for all but finitely many $\pi\in\widehat G_\tau$ then $\Gamma$ and $\Gamma'$ are $\tau$-representation equivalent  in $G$ (i.e.\ $n_{\Gamma}(\pi)=n_{\Gamma'}(\pi)$ for all $\pi\in\widehat G_\tau$). 
\end{theorem}

A key point in the proof of the previous theorem is the fact that $\widehat G_\tau$ can be always written as a (non-necessarily finite) union of strings (see Lemma~\ref{strings}).	
Under the stronger assumption that $\widehat G_\tau$ can be written as a finite union of strings, we obtain the following stronger result.

\begin{theorem}\label{thm1:finiteSMOTtau}
Let $G$ be a compact connected semisimple Lie group, let $K$ be a closed subgroup of $G$,  and let $\tau$ be a finite dimensional representation of $K$.
Assume there are $G$-integral dominant weights $\Lambda_{0,i}$ and  $\dir_i$ for $1\leq i\leq m$ such that 
\begin{equation}\label{eq:hypothesis}
\widehat G_\tau = \bigcup_{i=1}^m \; \{ \pi_{\Lambda_{0,i}+k\dir_i} :k\geq0 \}.
\end{equation}
Given an integer $q>0$, let $\widehat F_\tau$ be any finite subset of $\widehat G_\tau$ such that $\mathcal A_i:=\{k\in\N_0: \pi_{\Lambda_{0,i}+k\dir_i}\in \widehat F_\tau\}$ satisfies \eqref{eq1:finitecondition} for each $1\leq i\leq m$. 
Then, for any  finite subgroup $\Gamma$  of $G$ with $q$ divisible by $|\Gamma|$,  the finite set of multiplicities $n_{\Gamma}(\pi)$ for $\pi\in\widehat F_\tau$ determine the $n_{\Gamma}(\pi)$ for all $\pi\in\widehat G_\tau$. 
In particular, if $\Gamma$ and $\Gamma'$ are finite subgroups of $G$ with $q$ divisible by $|\Gamma|$ and $|\Gamma'|$ such that $n_{\Gamma}(\pi)=n_{\Gamma'}(\pi)$ for all   $\pi\in\widehat F_{\tau}$, then $\Gamma$ and $\Gamma'$ are $\tau$-representation equivalent in $G$.
\end{theorem}

\begin{remark}
The assumption in the previous theorem holds for Gelfand pairs of rank one, in particular for compact symmetric spaces of real rank one (see Remark~\ref{rem:finiteunion}). 
\end{remark}

A motivation of the authors for the questions treated in this paper are the applications to spectral geometry of locally homogeneous spaces. 
In fact, Theorem~\ref{thm1:finiteSMOTtau} is a main tool in \cite{LM-repequiv2}, where the authors study the strong multiplicity one property for the $\tau$-spectrum of a space covered by a compact symmetric space of real rank one and the connection between $\tau$-isospectrality and $\tau$-representation equivalence.

The article is organized as follows.
In Section~\ref{sec:generatingfunctions} we show that the generating function $F_{\dir,\Lambda_0,\Gamma}(z)$ is a rational function. 
The strong multiplicity one theorems, including 
Theorems~\ref{thm1:SMOTtau} and \ref{thm1:finiteSMOTtau}, are proved in Section~\ref{sec:SMOT}.

\subsection*{Acknowledgments}

The authors wish to thank T.N.~Venkataramana for very useful comments (via MathOverflow) regarding the proof of Lemma~\ref{strings}.

\section{Generating functions of strings} \label{sec:generatingfunctions}
Let $G$ be a compact connected semisimple Lie group, let $T$ be a maximal torus in $G$ and let $W=W(G,T)$ denote the corresponding Weyl group. 
Let $\Phi=\Phi(G,T)$ be the associated root system and let $\PP(G)$ be the lattice of $G$-integral weights. 
We fix a system of positive roots $\Phi^+=\Phi^+(G,T)$ and we let $\PP^{+}(G)$ be the set of dominant $G$-integral weights.
If $(\pi, V_\pi)$ is a finite dimensional representation of $G$,  denote by $\chi_\pi$ the character of $\pi$,  $\chi_\pi (g)= \tr \pi(g)$, for any $g\in G$. 
If $\Lambda\in \PP^{+}(G)$, let $\pi_{\Lambda}$ be the irreducible representation of $G$ with highest weight $\Lambda$.

\begin{definition}\label{def2:strings}
For $\dir,\Lambda_0\in \PP^{+}(G)$, we call the ordered set $\String{\dir, \Lambda_0} := \{\pi_{\Lambda_0+k\dir}:k\in \N_0\}$ the \emph{string of representations} associated to $(\dir,\Lambda_0)$.
The elements $\dir$ and $\Lambda_0$ will be called the \emph{direction} and the \emph{base} of the string respectively. 
\end{definition}

Let $\Gamma$ be a discrete, hence finite, subgroup of $G$. 
By Frobenius reciprocity, one has $n_{\Gamma}(\pi)=\dim V_\pi^\Gamma$ for any $\pi\in\widehat G$, where $V_\pi^\Gamma$ stands for the subspace of $V_{\pi}$ invariant by $\Gamma$. 
In order to study these numbers, we will encode them in a generating function  for representations lying in a string. 
More precisely, given a $(\dir,\Lambda_0)$-string $\String{\dir,\Lambda_0}$, we define the generating function  
\begin{equation}\label{eq2:F_Gamma(z)}
F_{\dir,\Lambda_0,\Gamma}(z) = \sum_{k\in \N_0} \dim V_{\pi_{\Lambda_k}}^\Gamma \, z^k,
\end{equation}
where $\Lambda_k=\Lambda_0+k\dir$ for any $k\geq0$.
From now on we will abbreviate $F_{\Gamma} (z) = F_{\dir,\Lambda_0,\Gamma}(z)$.

The next result is the main goal of this section. It  will be a main tool in the sequel.

\begin{proposition}\label{prop2:F_Gamma-rational}
Write $q=|\Gamma|$.
In the notation above,  there exists a complex polynomial $p(z)=p_{\dir,\Lambda_0,\Gamma}(z)$ of degree less than $q(\NN +1)$ such that
\begin{equation*}
  F_{\Gamma}(z) = \frac{p(z)}{(1-z^q)^{\NN+1}}.
\end{equation*}
\end{proposition}

The rest of this section will be devoted to give a proof of this result.
One has that
\begin{equation}\label{eq2:F_Gamma(z)2}
F_{\Gamma}(z)
    = \frac{1}{|\Gamma|}\sum_{\gamma\in\Gamma}  \sum_{k\in \N_0} \chi_{\pi_{\Lambda_k}}(\gamma)  z^k,
\end{equation}
since $\dim V_{\pi_{\Lambda_k}}^\Gamma = \frac{1}{|\Gamma|}\sum_{\gamma\in\Gamma} \chi_{\pi_{\Lambda_k}}(\gamma)$.
Our strategy will be to compute the terms $\chi_{\pi_{\Lambda_k}}(\gamma)$ for each $\gamma\in\Gamma$ by using a version of the Weyl character formula (see \cite[\S2.2]{ChenevierRenard15}).
To state this result we need to introduce some more notation.

For fixed $t\in T$, we let $Z=Z_t=C_G(t)^0$ be the identity component of the centralizer of $t$ in $G$, that is a compact subgroup of $G$.
Let $\Phi(Z,T)$ be the root system associated to $(Z,T)$, and let $\Phi_Z^+=\Phi^+\cap \Phi(Z,T)$, $W^Z=\{\sigma\in W: \sigma^{-1}\Phi_Z^+\subset \Phi^+\}$, $\rho=\frac12\sum_{\alpha\in\Phi^+}\alpha$, $\rho_Z=\frac12\sum_{\alpha\in\Phi_Z^+}\alpha$, and
\begin{equation}\label{eq2:P_Z(v)}
p_Z(v) = \prod_{\alpha\in\Phi_Z^+} \frac{\langle\alpha, v+\rho_Z\rangle}{\langle\alpha, \rho_Z\rangle}.
\end{equation}
If $t=\exp(H)\in T$ and $\mu$ is a weight, write $t^\mu= e^{\mu(H)}$.
Now, the \emph{Weyl character formula} in \cite[Prop.~2.3]{ChenevierRenard15} tells us that, for any $\Lambda\in \PP^{+}(G)$ and $t\in T$, one has 
\begin{equation}\label{eq2:chipiLambdat}
\chi_{\pi_\Lambda} (t) = \frac{\displaystyle\sum_{\sigma\in W^Z} \varepsilon(\sigma) \; t^{\sigma(\Lambda+\rho)-\rho} \; p_Z(\sigma(\Lambda+\rho)-\rho_Z)}{\displaystyle\prod_{\alpha\in\Phi^+\smallsetminus \Phi_Z^+}(1-t^{-\alpha})}.
\end{equation}

We now proceed with the proof of  Proposition~\ref{prop2:F_Gamma-rational}. 
For each $\gamma \in \Gamma$ we fix $t_\gamma \in T$, conjugate to $\gamma$, and we abbreviate $\Phi_\gamma^+=\Phi_{Z_\gamma}^+$. 
By substituting \eqref{eq2:chipiLambdat} in \eqref{eq2:F_Gamma(z)2}, we obtain that 
\begin{equation}\label{eq2:F_Gamma(z)3}
F_{\Gamma}(z)
    = \frac{1}{|\Gamma|}\sum_{\gamma\in\Gamma}
    \Big({ \prod_{\alpha\in\Phi^+ \smallsetminus \Phi_Z^+}(1-t_\gamma^{-\alpha})^{-1}}\Big
	)
    \sum_{\sigma\in W^Z} \varepsilon(\sigma)
    \,c_{\gamma,\sigma}(z),
\end{equation}
where
\begin{equation}\label{eq2:a_gamma,sigma(z)}
c_{\gamma,\sigma}(z) = \sum_{k\in \N_0}  {t_\gamma}^{\sigma(\Lambda_k+\rho)-\rho} \,p_Z(\sigma(\Lambda_k+\rho)-\rho_Z)\,  z^k.
\end{equation}
The next goal will be to show that $c_{\gamma,\sigma}(z)$ is a particular rational function.

\begin{lemma}\label{lem2:c_gamma,sigma}
In the notation above, we have that
\begin{equation}
c_{\gamma,\sigma}(z) = \frac{p_{\gamma,\sigma}(z)}{ \left(1-t_\gamma^{\sigma(\dir)}z\right)^{\NNg +1}}, 
\end{equation}
for some complex polynomial $p_{\gamma,\sigma}(z)$ of degree $\leq \NNg$.
\end{lemma}

\begin{proof}
On the one hand,
\begin{align}\label{eq2:p_Z(sigma(Lambda_k+rho))}
p_Z(\sigma(\Lambda_k+\rho)-\rho_Z)
    &= \prod_{\alpha\in\Phi_Z^+} \frac{\langle\alpha, \sigma(\Lambda_k+\rho)\rangle}{\langle\alpha, \rho_Z\rangle}
     = \prod_{\alpha\in\Phi_Z^+} \frac{\langle\sigma(\alpha), k\dir+\Lambda_0+\rho\rangle}{\langle\alpha, \rho_Z\rangle}\\
    &= \Big(\prod_{\alpha\in\Phi_Z^+} \frac{1}{\langle \alpha,\rho_Z\rangle}\Big)
    \Big(\prod_{\alpha\in\Phi_Z^+}
    \big(k\langle \sigma(\alpha),\dir\rangle +\langle \sigma(\alpha),\Lambda_0+ \rho\rangle\big)\Big),
    \notag
\end{align}
which is a polynomial in $k$ of degree $\NNg$ with complex coefficients.
We thus write $$p_Z(\sigma(\Lambda_k+\rho)-\rho_Z) = \sum_{j=0}^{\NNg} b_j {k}^j.$$

On the other hand, we have that $t_\gamma^{\sigma(\Lambda_k+\rho)-\rho} =t_\gamma^{\sigma(\Lambda_0+\rho)-\rho}\, t_\gamma^{k\sigma(\dir)}$. 
Hence
\begin{equation}
c_{\gamma,\sigma}(z)
    = t_\gamma^{\sigma(\Lambda_0+\rho)-\rho}\sum_{j=0}^{\NNg} b_j
    \sum_{k\in \N_0}k^j (t_\gamma^{\sigma(\dir)}z)^k.
\end{equation} 
We claim that for any $0\leq j\leq \NNg$,
\begin{equation}\label{eq2:claim}
\sum_{k\geq0}k^j t_\gamma^{k\sigma(\dir)}z^k = \frac{p_j(z)}{\left(1-t_\gamma^{\sigma(\dir)
}z\right)^{\NNg+1}}
\end{equation}
for some polynomial $p_j(z)$ of degree at most $\NNg$. Once this is done, the lemma clearly follows.

We next prove our claim.
The assertion is clear for $j=0$, by using the geometric series.

Assume now that \eqref{eq2:claim} holds for some $j<\NNg$.
Since $j!\binom{k+j}{k}=k^j + \sum_{l=0}^{j-1}c_lk^l$ for some $c_l\in \Z$, we obtain that
\begin{align}
\sum_{k\geq0} k^j y^k
    &= j!\sum_{k\geq0} \binom{k+j}{k} y^k - \sum_{l=0}^{j-1} c_l\sum_{k\geq0} k^l y^k
    = \frac{j!}{(1-y)^{j+1}}- \sum_{l=0}^{j-1} c_l\sum_{k\geq0} k^l y^k.
\end{align}
Now, since $(1-y)^{-(j+1)} = \dfrac{(1-y)^{\NNg-j}}{(1-y)^{\NNg +1}}$, by substituting $y=t_\gamma^{\sigma(\dir)}z$, then \eqref{eq2:claim} follows by induction.
\end{proof}

Thus, since $t_\gamma^q=\gamma^q=1$,  $t_\gamma^{\sigma(\dir) }$ is a root of unity of order a divisor of $q$, then Lemma~\ref{lem2:c_gamma,sigma} yields 
\begin{align} \label{eq2:c_gamma,sigma(z)2}
c_{\gamma,\sigma}(z)
    &= \frac{p_{\gamma,\sigma}(z)}{ \left(1-t_\gamma^{\sigma(\dir)}z\right)^{\NNg+1}}
    = \frac{\displaystyle p_{\gamma,\sigma}(z) \, \left(1-z^q\right)^{\NN -\NNg}\,\prod_{\xi^q=1,\,\xi\neq t_\gamma^{\sigma(\dir)}} (1-\xi z)^{\NNg+1}}{ \left(1-z^q\right)^{\NN+1}}.
\end{align}
Since $\NN\geq\NNg$ for all $\gamma\in\Gamma$, the degree of the polynomial in the numerator is less than or equal to 
$$
\NNg + q(\NN-\NNg) +(q-1)(\NNg+1)  < q(\NN+1),
$$
hence Proposition~\ref{prop2:F_Gamma-rational} follows from \eqref{eq2:F_Gamma(z)3} and \eqref{eq2:c_gamma,sigma(z)2}.

In the expression for $F_\Gamma(z)$ as a rational function in  Proposition~\ref{prop2:F_Gamma-rational}, the numerator $p(z)$ and the denominator $(1-z^q)^{\NN+1}$ usually share some factors. 
We next discuss a simple example illustrating this situation. 

\begin{example}\label{ex2:SU(2)}
Let $G=\SU(2)$, which has rank one. 
In this case, $\Phi^+=\{\alpha:=\varepsilon_1- \varepsilon_2\}$, $\widehat G = \{\pi_k:=\pi_{\alpha k/2}:k\in\N_0\}$,  $\dim V_{\pi_{k}}=k+1$, and every weight space in $V_{\pi_{k}}$ is one-dimensional having the form $\frac{k-2h}{2}\alpha$ for some $0\leq h\leq k$. 
Fix an even positive integer $q$ and set 
\begin{equation}
\Gamma=\left\{
\begin{pmatrix}
e^{2\pi i h/q}\\ & e^{-2\pi i h/q}
\end{pmatrix}: h\in\Z\right\}.
\end{equation}

We claim that $\dim V_{\pi_{k}}^\Gamma = 1+2\lfloor \tfrac{k}{q}\rfloor$ for any $k\geq0$.
Indeed, since $\Gamma$ acts by scalars in each weight space and these scalars are $q$-th roots of unity, we have that $\dim V_{\pi_{k}}^\Gamma$ is given by the number of weights in $V_{\pi_{k}}$ of the form $\frac{k-2h}2 \alpha$ for $0\leq h\leq k$ divisible by $q$. 
Consequently,
\begin{align}
F_\Gamma(z) 
&=\sum_{k\geq0} (1+\lfloor \tfrac{k}{q}\rfloor )\,z^k 
 =\sum_{j=0}^{q-1}\sum_{m\geq0} (1+\lfloor \tfrac{mq+j}{q}\rfloor )\,z^{mq+j} \\ \notag
&=\frac{1-z^q}{1-z} \sum_{m\geq 0} (1+2m)\,z^{qm} 
=\frac{1+z^q}{(1-z)(1-z^q)} 
=\frac {1+z+\dots+z^{2q-1}}{(1-z^q)^2}.
\end{align}
Then $\frac{1-z^{2q}}{1-z}$ and $(1-z^q)^2$ are the numerator and denominator of $F_\Gamma(z)$ respectively, as stated in Proposition~\ref{prop2:F_Gamma-rational} and the polynomial $\frac{1-z^q}{1-z} =1+z+\dots+z^{q-1}$ is their greatest common divisor. 

Kostant~\cite{Kostant85} computed explicitly $F_{\Gamma}(z)$ for the (obvious) string $\{\pi_k:k\geq0\}$ in $\SU(2)$ as in Example~\ref{ex2:SU(2)} for every finite subgroup $\Gamma$ of $\SU(2)$. 
\end{example}

\section{Strong multiplicity one theorems} \label{sec:SMOT}
The goal of this section is to prove Theorems~\ref{thm1:SMOTtau} and \ref{thm1:finiteSMOTtau}.
We first recall the context: $G$ is a compact connected semisimple Lie group, $K$ a closed subgroup of $G$, and $\Gamma$ a finite subgroup of $G$.
Furthermore, we fix a maximal torus $T$ in $G$ and a positive system in the associated root system.  
We denote by $\Phi^+$ the set of positive roots.
By Frobenius reciprocity, we have that
\begin{equation}\label{eq3:Frobenius}
n_\Gamma(\pi) = 
\dim \Hom_G(\pi, L^2(\Gamma\ba G)) 
=\dim \Hom_\Gamma(1_\Gamma, \pi|_K) 
=\dim V_\pi^\Gamma
\end{equation}
for every $\pi\in\widehat G$.

The main tool in what follows will be the generating function associated to a string of representations (see \eqref{eq2:F_Gamma(z)}) and the rational expression obtained in  Proposition~\ref{prop2:F_Gamma-rational}.

\begin{proposition} \label{lem3:SMOTstrings}
Let $\Gamma$ be a finite subgroup of $G$, let $q$ be any positive integer divisible by $|\Gamma|$, and let $\Lambda_0$ and $\dir$ be $G$-integral dominant weights.
If $\mathcal A\subset \N_0$ satisfies that 
\begin{equation}\label{eq3:finitecondition}
|\mathcal A\cap (j+q\Z)| \geq \NN+1  \quad \text{for all } 0\leq j\leq q-1, 
\end{equation}
then the set of multiplicities $n_{\Gamma}(\pi_{\Lambda_0+k\dir})$ for $k \in \mathcal A$ determine the whole sequence $n_{\Gamma}(\pi_{\Lambda_0+k\dir})$ for $k\in \N_0$.	
Moreover, for any $k\in \N_0$,  $n_{\Gamma}(\pi_{\Lambda_0+k\dir})$ can be linearly expressed in terms of the $n_\Gamma (\pi_{\Lambda_h})$ for
$h\in \mathcal A$.
In particular, if $\Gamma$ and $\Gamma'$ are finite subgroups of $G$ with $q$ divisible by $|\Gamma|$ and $|\Gamma'|$ such that $n_{\Gamma}(\pi_{\Lambda_0+k\dir}) = n_{\Gamma'}(\pi_{\Lambda_0+k\dir})$ for all $k\in \mathcal A$, then this equality holds for every $k\geq0$.
\end{proposition}

\begin{proof}
Set $\Lambda_k=\Lambda_0+k\dir$ for any $k\geq0$. By  Proposition~\ref{prop2:F_Gamma-rational} and \eqref{eq3:Frobenius}, there exists a polynomial $p(z) = p_{\dir, \Lambda_0, \Gamma}(z)$ of degree less than $q(\NN+1)$ such that
$$
F_{\Gamma}(z)=\sum_{k\geq0} n_{\Gamma}(\pi_ {\Lambda_k}) z^k = \frac{p(z)}{(1-z^q)^{\NN +1}}.
$$
We write $p(z) = \sum_{k=0}^{q(\NN +1)-1}b_k z^k$.
We first show that one has an expression 
\begin{equation}\label{eq3:n_lq+j}
n_{\Gamma}(\pi_ {\Lambda_{mq+j}}) = \sum_{h=0}^{\NN } b_{hq+j} \tbinom{m-h+\NN }{\NN }
\end{equation}
for any $0\leq j<q$ and for any $m\geq0$. Here $\tbinom{m-h+\NN }{\NN }=0$ for $h>m$, by convention, so the sum actually runs over $0\leq h\leq \min(\NN ,m)$.
To show \eqref{eq3:n_lq+j}, since $\frac{1}{(1-y)^{j+1}} = \sum_{k\geq0} \binom{k+j}{k} y^k$ for any $j \in \N_0$, then 
\begin{align*}
F_{\Gamma}(z)
    &= p(z)(1-z^q)^{-(\NN +1)} = \left(\sum_{j=0}^{q-1} \sum_{h=0}^{\NN } b_{hq+j} z^{hq+j} \right) \left(\sum_{k\geq0} \tbinom{k+\NN }{\NN } z^{kq}\right) \\
    &= \sum_{j=0}^{q-1} \sum_{m\geq0}  \left(\sum_{h=0}^{\NN } b_{hq+j} \tbinom{m-h+\NN }{\NN }\right) z^{mq+j}.
\end{align*}

We fix $0\leq j<q$. 
Since $|\mathcal A\cap (j+q\Z)|\geq \NN+1$, there are $m_0,\dots,m_{\NN }$ different non-negative integers such that $m_iq+j\in\mathcal A$ for all $i$.
The square matrix
$$
\Big(\tbinom{m_i-h+\NN }{\NN } \Big)_{i,h=0}^{\NN}
$$
is clearly non-singular, so the system of $\NN +1$ linear equations with $\NN +1$ unknowns $b_j,\dots,b_{q\NN +j}$,
$$
n_{\Gamma}(\pi_ {\Lambda_{m_iq+j}}) 
= \sum_{h=0}^{\NN }  \tbinom{m_i-h+\NN }{\NN }\,b_{hq+j}
\qquad\text{ for $i=0,\dots,\NN $,}
$$
has a unique solution.
Consequently, the coefficients
$b_{hq+j}$ for $0\leq h\leq \NN $ can be linearly expressed in terms of the multiplicities 
$n_{\Gamma}(\pi_ {\Lambda_k})$ for $k\in \mathcal A\cap (j+q\Z)$.

Since this holds for every $j$, we conclude that the $b_{k}$ for every $0\leq k< q(\NN+1)$ are determined by the $n_{\Gamma}(\pi_ {\Lambda_k})$'s for $k\in \mathcal A$, and hence also $p(z)$ as well as  $F_\Gamma(z)$, are determined. Therefore all the $\{n_\Gamma(\pi_k): k\in \N_0\}$ are  linearly determined. 
\end{proof}

Proposition~\ref{lem3:SMOTstrings} gives a finite strong multiplicity one result
for a string of representations. 
In the next result we show a refinement, by  again
proving strong multiplicity one, now  
valid with a possible infinite set of exceptions of sufficiently small density.

\begin{corollary}\label{cor3:density2Gammas}
Let $\Gamma,\Gamma'$ be finite subgroups of $G$, let $q$ be the least common multiple between $|\Gamma|$ and $|\Gamma'|$, and let $\Lambda_0$ and $\dir$ be $G$-integral dominant weights.
If
\begin{equation}\label{eq:densitydosGammas}
	\limsup_{t\to\infty} \frac{|\{0\leq k\leq t: n_{\Gamma}(\pi_{\Lambda_0+k\dir}) \neq n_{\Gamma'}(\pi_{\Lambda_0+k\dir})\}|}{t} <\frac{1}{q},
\end{equation}
then $n_{\Gamma}(\pi_{\Lambda_0+k\dir})= n_{\Gamma'}(\pi_{\Lambda_0+k\dir})$ for all $k\geq0$. 
In particular, if $n_{\Gamma}(\pi_{\Lambda_0+k\dir}) = n_{\Gamma'}(\pi_{\Lambda_0+k\dir})$ for all but finitely many $k\geq0$, then the equality holds for every $k\geq0$. 
\end{corollary}

\begin{proof}
Set as usual $\Lambda_k=\Lambda_0+ k\dir$ for any $k\geq0$.
We want to show that the subset of $\N_0$ $\mathcal A:=\{k\in \N_0: n_{\Gamma}(\pi_{\Lambda_k}) = n_{\Gamma'}(\pi_{\Lambda_k})\}$  satisfies \eqref{eq3:finitecondition}.
This done, Proposition~\ref{lem3:SMOTstrings} ensures that $n_{\Gamma}(\pi_{\Lambda_k})= n_{\Gamma'}(\pi_{\Lambda_k})$ for all $k\geq0$ as required.  

By the assumption \eqref{eq:densitydosGammas},  $\displaystyle \liminf_{t\to\infty} \frac{|\{0\leq k\leq  t: k\in\mathcal A\}|}{t} > \frac{q-1}{q}$. 
This means that for every $\varepsilon>0$ there is $r_0>0$ such that
\begin{equation}
\frac{|\{k\in\mathcal A:  k< rq(\NN +1)\}|}{rq(\NN +1)} > \frac{q-1+\varepsilon}{q}\qquad \text{for all }r\geq r_0,
\end{equation}
or equivalently,
\begin{equation}\label{eq3:densityproof2}
|\{k\in\mathcal A:  k\leq   rq(\NN +1)\}| > r(q-1+\varepsilon)(\NN+1)
\qquad \text{for all }r\geq r_0.
\end{equation}

Let $r\geq r_0$ and $j_0\in\Z$ satisfying $0\leq j_0<q$. 
We have that 
\begin{align}
	|\{k\in\mathcal A: k< rq(\NN+1)\}| 
	&= \sum_{j=0}^{q-1} |\{k\in\mathcal A: k< rq(\NN+1)\}\cap (j+q\Z)| \\
	&\leq |\mathcal A\cap (j_0+q\Z)| +  (q-1)r(\NN+1) . \notag
\end{align}
Hence, \eqref{eq3:densityproof2} implies that $|\mathcal A\cap (j_0+q\Z)|>r\varepsilon (\NN+1)$ for every $r\geq r_0$.
By taking any $r\geq 1/\varepsilon$ we obtain that $|\mathcal A\cap (j_0+q\Z)|>\NN+1$, as required.  
\end{proof}

We will need the following useful fact.

\begin{lemma}\label{strings}
For any $\tau \in \widehat K$, the set $\widehat G_\tau$ can be written as a union of strings having a common direction $\dir$, that is, 
$\widehat G_\tau= \bigcup_{\Lambda \in \mathcal Q_\tau} \{\pi_ {\Lambda+k\dir}:k\in \N_0\}$ for some subset $\mathcal Q_\tau$ of $\PP^{+}(G)$. 
\end{lemma}

\begin{proof}
The left-regular representation of $G$ on $L^2(G/K)$ decomposes 
$$
L^2(G/K) \simeq \bigoplus_{\pi\in\widehat G_{1_K}} (\dim V_\pi^K)\, V_\pi, 
$$
where $1_K$ denotes the trivial representation of $K$.
Let $\dir$ be the highest weight of any non-trivial representation in $\widehat G_{1_K}$. 
Then, for any $\pi_\Lambda\in\widehat G_\tau$, one has that $\pi_{\Lambda+k\dir} \in \widehat G_\tau$ for all $k\geq0$ (see for instance \cite[Thm.~3.9]{Kostant04}).
That is, $\String{\dir,\Lambda} \subset \widehat G_\tau$, and consequently, 
\begin{equation}
\widehat G_\tau = \bigcup_{\Lambda\in \PP^{+}(G): \, \pi_\Lambda\in \widehat G_\tau} \String{\dir,\Lambda},
\end{equation}
which completes the proof.
\end{proof}

We are now in a position to give the proofs of the main theorems. 

\begin{proof}[Proof of Theorem~\ref{thm1:SMOTtau}]
Let $q$ be the least common multiple between $|\Gamma|$ and $|\Gamma'|$. 
Since the set of possible exceptions to equality of multiplicities in $\widehat G_\tau$ is finite, the set of exceptions is also finite in any string associated to $(\Lambda_0,\dir)$.
Now, Corollary~\ref{cor3:density2Gammas} implies that $n_{\Gamma}(\pi) = n_{\Gamma'}(\pi)$ for every $\pi$ in each string, thus equality holds for all $\pi\in\widehat G_\tau$ by Lemma~\ref{strings}. 
\end{proof}

\begin{proof}[Proof of Theorem~\ref{thm1:finiteSMOTtau}]
We first recall the assumptions. 
We have that 
\begin{equation*}
\widehat G_\tau = \bigcup_{i=1}^m \; \{ \pi_{\Lambda_{0,i}+k\dir_i} :k\geq0 \},
\end{equation*}
for some  $G$-integral dominant weights $\Lambda_{0,i}$ and $\dir_i$, for $1\leq i\leq m$.  
Furthermore, $n_{\Gamma}(\pi)=n_{\Gamma'}(\pi)$ for all   $\pi\in\widehat F_{\tau}$, where $\widehat F_\tau$ is a finite subset of $\widehat G_\tau$ such that for each $1\leq i\leq m$ the subset $\mathcal A_i=\{k\in\N_0: \pi_{\Lambda_{0,i}+k\dir_i}\in \widehat F_\tau\}$ of $\N_0$ satisfies \eqref{eq3:finitecondition}.

Fix $1\leq i\leq m$.
From what has been assumed, it follows that $n_{\Gamma}(\pi_{\Lambda_{0,i} + k\dir_i})= n_{\Gamma'}(\pi_{\Lambda_{0,i} + k\dir_i})$ for all $k\in\mathcal A_i$.
Proposition~\ref{lem3:SMOTstrings} forces  $n_{\Gamma}(\pi_{\Lambda_{0,i} + k\dir_i})= n_{\Gamma'}(\pi_{\Lambda_{0,i} + k\dir_i})$ for all $k\geq0$. 
We conclude that $n_{\Gamma}(\pi)= n_{\Gamma'}(\pi)$ for all $\pi\in\widehat G_\tau$, that is, $\Gamma$ and $\Gamma'$ are $\tau$-representation equivalent in $G$.
\end{proof}

\begin{remark} \label{rem:finiteunion}
We end the article by giving some references  in the literature of explicit expressions of $\widehat G_\tau$ as a union of strings. 
Explicit descriptions of the set $\widehat G_\tau$ have been used for different purposes, which makes the task of providing a complete list of references difficult. 

We first assume that $G/K$ is a compact Riemannian symmetric space with $G$ and $K$ connected. 
Let $G'/K$ be the non-compact dual space, let $G'=KAN$ be an Iwasawa decomposition of $G'$, and let $M$ be the centralizer of $A$ in $K$. 
We denote by $\mathfrak g',\mathfrak m,\mathfrak a$ the Lie algebras of $G',M,A$ respectively.
Let $\mathfrak b$ be a Cartan Subalgebra of $\mathfrak m$ such that $\mathfrak b_\C\oplus \mathfrak a_\C$ is a Cartan subalgebra of $\mathfrak g_\C$. 
Let $1_K$ denote the trivial representation of $K$. 

The Cartan--Helgason theorem (see for instance 
\cite[Thm.~9.70]{Knapp-book-beyond}) ensures that $\pi_\Lambda \in\widehat G_{1_K}$ if and only if $\Lambda|_{\mathfrak b}= 0$ and ${\langle \Lambda, \alpha\rangle}/ {\langle \alpha, \alpha\rangle}\in\N_0$ for every positive restricted root $\alpha \in \Phi^+(\mathfrak g',\mathfrak a)$. 
This implies that $\widehat G_{1_K}$ can be written as a disjoint union of strings of representations with the same direction $\dir$ (there are several choices for $\dir$). 
When $G/K$ has real rank one, $\widehat G_{1_K}$ is exactly one string and furthermore, Camporesi~\cite[Thm.~2.4]{Camporesi05JFA} proved (by using \cite{Kostant04}) that there is $\Lambda_{\sigma,\tau}\in\PP^{+}(G)$ for each $(\sigma,\tau)\in\widehat M\times\widehat K$ such that 
\begin{equation}\label{eq:Camporesi}
\widehat G_\tau = \bigcup_{\sigma\in\widehat M: \,\Hom_M(\sigma, \tau|_M)\neq0 } \String{\dir,\Lambda_{\sigma,\tau}} ,
\end{equation}
for any $\tau\in\widehat K$ (see also \cite{Camporesi05Pacific}). 
We note that the union in \eqref{eq:Camporesi} is finite.
Formula \eqref{eq:Camporesi} can be seen as a generalization of the Cartan--Helgason theorem for an arbitrary $\tau\in\widehat K$.

Heckman and van Pruijssen~\cite{HeckmanPruijssen16} extended the previous results to Gelfand pairs of rank one.
In addition to the compact symmetric spaces of real rank one, these spaces include $(G,K)=(\textrm{G}_2, \SU(3))$ and $(\Spin(7), \textrm{G}_2)$. 
\end{remark}

The authors do not expect there are many other situations where the condition \eqref{eq:hypothesis} in Theorem~\ref{thm1:finiteSMOTtau} holds than those given in Remark~\ref{rem:finiteunion}.
We next give two simple instances where this is not the case, i.e.\ $\widehat G_\tau$ cannot be written as a finite union of strings. 

\begin{remark}
Let $G$ be any compact connected semisimple Lie group. 
By setting $K=\{e\}$, one clearly has $\widehat G_{1_K} = \widehat G$. 
We claim that, if the rank $n$ of $G$ is at least $2$ (i.e.\ any $G$ aside of $\SO(3)$ and $\SU(2)$), then $\widehat G$ cannot be written as a finite union of strings. 
Let $\varpi_1,\dots, \varpi_n$ be the fundamental weights of $\Phi^+(\mathfrak g_\C,\mathfrak t_\C)$, thus any element in $\mathcal P^+(G)$ is of the form $a_1\varpi_1+\dots+ a_n\varpi_n$ for some $a_1,\dots,a_n\in \N_0:=\N\cup\{0\}$. 

We consider the subset of $\widehat G$ given by 
\begin{equation*}
\mathcal G:= \widehat G \cap \{\pi_{m_1\varpi_1+m_2\varpi_2}: m_1,m_2\in\N\}, 
\end{equation*}
which clearly has infinitely many elements. 
One has that $\String{\omega,\Lambda_0}\cap \mathcal G$ is finite unless $\omega,\Lambda_0 \in \N_0\varpi_1\oplus \N_0\varpi_2$.
It is clear that we cannot cover $\mathcal G$ with finitely many strings of the form $\String{\omega,\Lambda_0}$ with $\omega\in \N\varpi_2$ and $\Lambda_0\in \N_0\varpi_1\oplus \N_0\varpi_2$.
Moreover, given a string $\String{a\varpi_1+b\varpi_2,c\varpi_1+d\varpi_2}$ with $a,b,c,d\in\N_0$ and $a>0$, the highest weight of an irreducible representation in it is of the form
\begin{equation*}
k(a\varpi_1+b\varpi_2)+(c\varpi_1+d\varpi_2)= (ak+c)\varpi_1+ (bk+d)\varpi_2,
\end{equation*}
thus the quotient
\begin{equation*}
\frac{bk+d}{ak+c} \leq  \frac{bk+d}{ak} <\frac{b}{a} +d
\end{equation*}
is bounded for all $k\in\N$.
Thus, this string cannot reach any element $\pi_{m_1\varpi_1+m_2\varpi_2}\in \mathcal G$ such that $m_2/m_1>b/a +d$.
It follows that $\mathcal G$ cannot be covered by finitely many strings.

A similar situation holds for compact symmetric spaces $G/K$ of real rank at least $2$. 
For instance, we set $G=\SO(n)$ and $K=\SO(n-2)\times\SO(2)$ for some $n\geq5$. 
Under the standard choice of $\Phi^+(G,K)$ (e.g.\ as in \cite[\S{}C.1]{Knapp-book-beyond}), we get 
\begin{equation}
\widehat G_{1_K}= \{\pi_{a\varpi_1+b\varpi_2}: a,b\in\N_0\},
\end{equation}
where $\varpi_1,\varpi_2,\dots$ denote the fundamental weights. 
Now, it is clear that the same argument as above shows that $\widehat G_{1_K}$ cannot be written as a finite union of strings. 
\end{remark}

\bibliographystyle{plain}

\end{document}